\documentclass[12pt,centertags,oneside]{amsart}

\usepackage{amsmath,amstext,amsthm,amscd,typearea,hyperref}
\usepackage{amssymb}
\usepackage{a4wide}
\usepackage[mathscr]{eucal}
\usepackage{mathrsfs}
\usepackage{typearea}
\usepackage{charter}
\usepackage{pdfsync}
\usepackage[a4paper,width=16.2cm,top=3cm,bottom=3cm]{geometry}

\numberwithin{equation}{section}

\newtheorem{theorem}{Theorem}
\newtheorem{lemma}[theorem]{Lemma}
\newtheorem{corollary}[theorem]{Corollary}
\newtheorem{proposition}[theorem]{Proposition}

\theoremstyle{definition}
\newtheorem{definition}[theorem]{Definition}
\newtheorem{example}[theorem]{Example}
\newtheorem{remark}[theorem]{Remark}

\numberwithin{theorem}{section}
\numberwithin{equation}{section}

\newcommand{\R}{\mathbb{R}}
\newcommand{\C}{\mathbb{C}}
\newcommand{\Cc}{\mathcal{C}}

\newcommand{\Om}{\Omega}
\newcommand{\Omf}{\partial\Omega}
\newcommand{\Omb}{\bar{\Omega}}

\newcommand{\fii}{\varphi}

\newcommand{\U}{\mathtt{U}}

\newcommand{\ca}{\text{Cap}}

\DeclareMathOperator{\diam}{diam}

\DeclareMathOperator{\dist}{dist}

\makeatletter
\@namedef{subjclassname@2010}{%
  \textup{2010} Mathematics Subject Classification}
\makeatother

\begin{document}

\title[Regularity of solutions to the Dirichlet problem]{Regularity of solutions to the Dirichlet problem for  Monge-Amp\`ere equations}

\author{Mohamad Charabati}
\date{\today}

%\renewcommand{\thefootnote}{}
%\footnote{2010 \emph{Mathematics Subject Classification}: 32W20, 32U15.}
%\footnote{\emph{Key words and phrases}: complex Monge-Amp\`ere equation, Hausdorff-Riesz measure, strongly hyperconvex Lipschitz domain, plurisubharmonic function.}
%\renewcommand{\thefootnote}{\arabic{footnote}}
%\setcounter{footnote}{0}

\begin{abstract}
We study  H\"older continuity of  solutions to the Dirichlet problem for  measures having density in $L^p$, $p>1$, with respect to  Hausdorff-Riesz measures of order $2n-2+\epsilon$ for  $0<\epsilon \leq 2$, in a bounded strongly hyperconvex Lipschitz domain and the boundary data belongs to $ \Cc^{0,\alpha}(\Omf)$, $ 0<\alpha \leq 1$. 
\end{abstract}



\maketitle

\section*{Introduction}
Let $\mu$ be a finite Borel measure on a bounded domain $\Om \subset \C^n$.
Given  $ f\in L^p(\Om,\mu)$ for $p>1$, and $ \fii \in \Cc(\Omf)$, the Dirichlet problem for the complex Monge-Amp\`ere equation asks for a function $u$  such that
\begin{center}
$  Dir(\Om,\fii,fd\mu): $
$ \left\{\begin{array}{ll}
            u\in PSH(\Omega)\cap\Cc(\Omb),& \\
           (dd^c u)^n=f d\mu  & \text{  in  } \; \Omega, \\
            u=\fii  & \text{ on  }  \Omf, \\
         \end{array} \right. $
\end{center}
where $ PSH(\Om)$ denotes the set of plurisubharmonic (psh) functions in $\Om$ and  
$ (dd^c.)^n $ is the complex  Monge-Amp\`ere operator.

This problem  has attracted attention for many years,
we refer the reader to \cite{BT76, Ce84, CP92,  Bl96, Ko98}, and references therein, for more details.

\medskip

In the case when $\Om$ is a  strongly pseudoconvex domain and $\mu$ is the Lebesgue measure,
Bedford and Taylor \cite{BT76} 
proved the existence and uniqueness of the solution to $ Dir(\Om,\fii,f dV_{2n})$  with   $0 \leq f \in \Cc(\Omb)$. 
Moreover, the solution is $ \alpha$-H\"older continuous when
$ \fii \in Lip_{2\alpha}(\Omf)$ and
$ f^{1/n} \in Lip_{\alpha}(\Omb)$ for $0< \alpha \leq 1$.
For more general domains, the existence and  regularity of the solution had been studied in \cite{Bl96, Ch15}.
\\
Ko{\l}odziej \cite{Ko96,Ko98} demonstrated that this problem still admit a unique solution  
When $  f \in L^p(\Om)$, $p>1$, and more generally when the right hand side of the complex Monge-Amp\`ere equation is a measure satisfying some sufficient condition which is close to be best possible.

H\"older continuity of  solutions to this problem for densities in $L^p$ with respect to the Lebesgue measure was studied in \cite{GKZ08}, \cite{N14}, \cite{Ch15} and \cite{BKPZ15}.
On a compact K\"ahler manifold, the existence of solutions is due to \cite{Ko98} and  H\"older regularity of  solutions to  complex Monge-Amp\`ere equations has been  investigated by many authors, we refer to \cite{ Ko08,  DDGPKZ14} for more details.

\smallskip
In the case of singular measures with respect to the Lebesgue measure,  H. H. Pham \cite{Ph10} proved the H\"older continuity of the solution to the complex Monge-Amp\`ere equation on a compact K\"ahler manifold.  There is no study about the regularity in the local case in $\C^n$.

Our purpose in this paper is to explore the H\"older continuity of the solution to the Dirichlet problem $Dir(\Om,\fii,fd\mu)$ when $\mu$ is a Hausdorff-Riesz measure of order $2n-2+\epsilon$, for $0<\epsilon \leq 2$ (see Definition \ref{def Hausdorff Riesz}). Precisely, we prove the following.

\begin{theorem}\label{main2 C11}
Let $\Om$ be a bounded strongly hyperconvex Lipschitz domain in $\C^n$ and $\mu$ be a Hausdorff-Riesz measure of order $2n-2+\epsilon$ for $0<\epsilon \leq 2$. Suppose that  $ \fii \in \Cc^{1,1}(\Omf)$ and $ 0 \leq f \in L^p(\Om,\mu)$ for some $p>1$, then the unique solution to $ Dir(\Om,\fii,fd\mu)$ is H\"older continuous on $\Omb$ of exponent $\epsilon \gamma/2$ for any $ 0< \gamma < 1/(nq+1)$ and $1/p+1/q=1$.
\end{theorem}

This result generalizes the one proved in \cite{GKZ08, Ch15}  from which the main idea  of our proof originates.

When the boundary data is merely H\"older continuous we state the regularity of the solution using the previous theorem.

\begin{theorem}\label{main2 C0,alpha}
Let $\Om$ be a bounded strongly hyperconvex Lipschitz domain in $\C^n$ and $\mu$ be a Hausdorff-Riesz measure of order $2n-2+\epsilon$ for $0<\epsilon \leq 2$. Suppose that $ \fii \in \Cc^{0,\alpha}(\Omf)$, $ 0<\alpha \leq 1$ and $ 0 \leq f \in L^p(\Om,\mu)$ for some $p>1$, then the unique solution to $ Dir(\Om,\fii,fd\mu)$ is H\"older continuous on $\Omb$ of exponent $ \frac{\epsilon}{\epsilon+6} \min \{ \alpha, \epsilon \gamma  \}$ for any $ 0< \gamma < 1/(nq+1)$ and $1/p+1/q=1$.

Moreover, when $\Om$  is a smooth strongly pseudoconvex domain the H\"older exponent of the solution will be
$ \frac{\epsilon}{\epsilon+2} \min \{ \alpha, \epsilon \gamma  \}$, for any $ 0< \gamma < 1/(nq+1)$.
\end{theorem}
In the case of the Lebesgue measure, i.e. $\epsilon=2$, in  a smooth bounded  strongly pseudoconvex domain  we get the H\"older exponent $\min \{ \alpha/2, \gamma\}$ which is better than the one obtained in \cite{BKPZ15}.
\medskip

Finally, a natural question is that if we have a H\"older continuous subsolution to the Dirichlet problem, can we get a H\"older continuous solution?\\
This question is still open in the local case (see \cite{DDGPKZ14} for a positive answer in the compact setting). However, we deal in Theorem \ref{main3} some particular case.

\noindent {\bf Acknowledgements}.
I am very grateful to my advisor, Professor Ahmed Zeriahi, for useful discussions and suggestions.

\section{Existence of solutions  to the Dirichlet problem}

This section is devoted to explain briefly the existence of continuous solutions to the Dirichlet problem $Dir(\Om,\fii,\mu)$ for  measures $\mu$ dominated by Bedford-Taylor's capacity, as in (\ref{measure}) below, in a bounded strongly hyperconvex Lipschitz domain.

We begin by recalling (see \cite{Ch15}) that a bounded domain $ \Om \subset \C^n$ is called {\it strongly hyperconvex Lipschitz}  domain if there exist an open neighborhood $\Om'$ of $\Omb$ and 
a Lipschitz  plurisubharmonic defining function $ \rho : {\Om}' \to \R $ 
 such that  
\begin{enumerate}
\item $\Om= \{ \rho < 0\}$ and $\partial \Omega = \{  \rho = 0\}$,
\item $ dd^c \rho \geq 2 \beta $ in  $\Omega$
in the weak sense of currents, where $\beta$  is the standard K\"ahler form on $\C^n$.
\end{enumerate}

\begin{example}\label{hyperconvex}
\
\begin{enumerate}
\item Any bounded  strictly convex domain is a bounded strongly hyperconvex Lipschitz domain. 
\item The nonempty finite intersection of smooth strongly pseudoconvex bounded domains in $ \C^n$  is a bounded strongly hyperconvex Lipschitz  domain. 
\item The domain $ \Om = \{ z = (z_1,\cdots,z_n) \in \C^n ; |z_1| + \cdots + |z_n| < 1 \}$
($n \geq 2$) is a bounded strongly hyperconvex Lipschitz domain in $\C^n$ 
with non-smooth boundary.
\end{enumerate}
\end{example}

We need in the sequel the following property of a bounded strongly hyperconvex Lipschitz domain.

\begin{lemma}\label{lemma tilde rho}
Let $\Om$ be a bounded strongly hyperconvex Lipschitz domain. Then there exist  a defining function $\tilde{\rho} \in PSH(\Om) \cap \Cc^{0,1}(\Omb)$ such that near $\Omf$ we have
\begin{equation}\label{tilde rho}
c \dist(z,\Omf) \geq - \tilde{\rho}(z) \geq \dist(z,\Omf)^2,
\end{equation}
for some $c>0$ depending on $\Om$.

Moreover $dd^c \tilde \rho \geq \beta$ in the weak sense of currents on $\Omega$.
\end{lemma}
\begin{proof}
 Since $\Om$ is a strongly hyperconvex Lipschitz domain, there exists a defining function $\rho$ such that $dd^ c \rho \geq 2 \beta$ in the weak sense of currents on $\Omega$. Let us fix $\xi \in \Omf$, then the function defined by 
$ \tilde{\rho}_{\xi}(z):= \rho(z)- |z-\xi|^2$ is Lipschitz continuous in $\Omb$ and satisfies $dd^c \tilde{\rho}_{\xi} \geq \beta$ in the weak sense of currents on $\Omega$. Hence  
$ \tilde{\rho}_{\xi} \in PSH(\Om)\cap \Cc^{0,1}(\Omf)$.  Set
$$
 \tilde{\rho}:= \sup\{ \tilde{\rho}_{\xi}; \; \xi \in \Omf \}.
$$
It is clear that $\tilde{\rho} \in \Cc^{0,1}(\Omb) \cap PSH(\Omega)$ and thus the first inequality in (\ref{tilde rho}) holds. For any $\xi \in \Omf$ we have
$ - \tilde{\rho}_{\xi} (z) \geq C |z-\xi|^2$, so we infer that 
$$
-\tilde{\rho}(z) \geq  \dist(z,\Omf)^2,
$$
for any $z$ near $\Omf$.

The last statement follows from the fact that for any $\xi \in \Omf$, $dd^c \tilde{\rho}_{\xi} \geq \beta$ in the weak sense of currents on $\Omega$.
\end{proof}

\begin{remark}\label{rho smooth}
When $\Om$ is a smooth strongly pseudoconvex domain, we know that the defining function $\rho$ satisfies near the boundary, 
$$
-\rho \approx \dist(.,\Omf).
$$
\end{remark}

\begin{definition}
A finite Borel measure $\mu$  on $\Om$ is said to satisfy Condition $ \mathcal{H}(\tau)$ for some fixed $\tau>0$ if  there exists a positive constant $A$  such that 
\begin{equation}\label{measure}
\mu(K) \leq A \ca(K,\Om)^{1+\tau},
\end{equation}
for any Borel subset $K$ of $\Om$.
\end{definition}

Ko{\l}odziej \cite{Ko98} demonstrated the existence of a continuous solution to $Dir(\Om,\fii,\mu)$  when $\mu$ verifies   (\ref{measure}) and some local extra condition in a bounded strongly pseudoconvex domain with smooth boundary. Furthermore, he disposed of the extra condition in \cite{Ko99} using Cegrell's result \cite{Ce98} about the existence of a solution in the energy class $\mathcal{F}_1$. 

Here, we only summarize the steps of the proof of the existence of continuous solutions to  $Dir(\Om,\fii,\mu)$ in a bounded strongly hyperconvex Lipschitz domain following the lines of Ko{\l}odziej and Cegrell's arguments in \cite{Ko98, Ce98}.

\begin{theorem}\label{main1}
Let $\mu$ be a  measure satisfying Condition $ \mathcal{H}(\tau)$, for some $\tau>0$,  on a bounded strongly hyperconvex Lipschitz domain  $\Om \subset \C^n$ and $ \fii \in \Cc(\Omf)$. Then there exists a unique continuous solution to $Dir(\Om,\fii,\mu)$.
\end{theorem}

\begin{proof}
Suppose first that $\mu$ has compact support in $\Om$. Let us consider a subdivision $I^s$ of $\text{supp}\mu$  consisting of $3^{2ns}$ congruent semi-open cubes $I_j^s$ with side $d_s=d/3^s$, where $ d:=\diam (\Om)$ and $ 1 \leq j \leq 3^{2ns}$.
Thanks to Proposition 5.3 in \cite{Ch15}, one can find $ u_s \in PSH(\Om)\cap \Cc(\Omb)$ such that
$$
u_s=\fii  \text{ on } \Omf,
$$
and
$$
(dd^c u_s)^n = \mu_s := \sum_j \frac{\mu(I_j^s)}{d_s^{2n}} \chi_{I_j^s} dV_{2n} \text{ in } \Om, 
$$ 
where $\chi_{I_j^s}$ is  the characteristic function of $I_j^s$.
We can ensure following \cite{Ko98} that 
there exist $s_0>0$ and  $B=B(n,\tau)>0$ such that 
for all $s>s_0$ the measure $\mu_s$ satisfies
$$
 \mu_s(K)  \leq B \ca(K,\Omega)^{1+\tau}, 
$$
for all Borel subsets $K$ of $ \Om$.
Then, we prove that the $L^\infty$-norm of $u_s$, for $s>s_0$, is bounded by an absolute constant depending only on $n$ and $\tau$.  

We set $ u:= (\limsup u_s)^*$ which is a candidate to be the solution to $Dir(\Om,\fii,\mu)$.
The delicate point is then to show that $ (dd^c u_s)^n$ converges to $(dd^c u)^n$ in the weak sense of measures  as in \cite{Ko98} (see also \cite{GZ07}). For this purpose, we invoke Cegrell's techniques \cite{Ce98} to ensure that 
$$
 \int_\Om u_s \, d\mu \to \int_\Om u \, d\mu \text{ and } \int_\Om |u_s -u| \, d\mu_s \to 0, \text{ when } s \to +\infty.
$$
For the general case,  let $\chi_j$ is a nondecreasing sequence of smooth cut-off function, $\chi_j \nearrow 1$ in $\Om$.  We get solutions $u_j$ to the Dirichlet problem for the measures $ \chi_j \mu$. Then,  the solutions $u_j$ are uniformly bounded. We set $u:=(\limsup u_j)^* \in PSH(\Om)\cap L^\infty(\Omb)$ and the last argument yields that $u$ is the required bounded solution to $Dir(\Om,\fii,\mu)$. 

Finally, we assert the continuity of the solution in the spirit of \cite{Ko98}.
\end{proof}

We introduce an important class of Borel measures on $\Om$ containing Riesz measures and
closely related to Hausdorff measures which play an important role in geometric
measure theory \cite{Ma95}. We call such measures Hausdorff-Riesz measures.

\begin{definition}\label{def Hausdorff Riesz}
A finite Borel measure on $\Om$ is called a { \it Hausdorff-Riesz } measure of order $2n-2+\epsilon$, for $0<\epsilon \leq 2$ if it satisfies the following condition : 
\begin{equation}\label{Hausdorff Riesz}
\mu(B(z,r)\cap \Omega) \leq C r^{2n-2+\epsilon}, \;\; \forall z\in \Omb,\; \forall 0<r<1,
\end{equation}
for some positive constant $C$. 
\end{definition}

We  give some interesting examples of  Hausdorff-Riesz measures. 
\begin{example}
\
\begin{enumerate}
\item The Lebesgue measure $dV$ on $\Om$, for $\epsilon=2$.
\item The surface measure of a compact real hypersurface, for $\epsilon=1$.
\item Measures of the type $ dd^c v \wedge \beta^{n-1}$, where $v$ is a $\alpha$-H\"older continuous subharmonic function in a neighborhood of $\Omb$, for $\epsilon=\alpha$.
\item The measure $ {\bf{1}}_E \mathcal{H}^{2n-2+\epsilon}$, where $ \mathcal{H}^{2n-2+\epsilon} $ is the Hausdorff measure and $E$ is a Borel set such that $ \mathcal{H}^{2n-2+\epsilon}(E) < +\infty$.
\item If $\mu$ is a Hausdorff-Riesz measure of order $2n-2+\epsilon$, then $f d\mu$ is a Hausdorff-Riesz measure of order $(2n-2+\epsilon)/q$, for any $ f \in L^p(\Om,\mu)$, $p > (2n-2+\epsilon)/\epsilon$.
\end{enumerate}

\end{example}

The existence of  continuous solutions to $Dir(\Om,\fii, f d\mu)$ for such  measures follows immediately from Theorem \ref{main1} and the following lemma.

\begin{lemma}\label{dominated by cap}
Let $ \Omega$ be a bounded strongly hyperconvex Lipschitz domain and $\mu$ be a Hausdorff-Riesz measure of order $2n-2+\epsilon$, for $0<\epsilon \leq 2$. Assume that $ 0 \leq f \in  L^p(\Om,\mu)$ for $ p > 1$, then for all $  \tau >1$ there exists $ D>0$ depending on $\tau, \epsilon$ and $\diam(\Om)$ such that for any Borel set $K \subset \Om$, 
\begin{equation}
\int_K f d\mu \leq D \|f\|_{L^p(\Om,\mu)} [\ca (K,\Om)]^\tau.
\end{equation} 
\end{lemma}

\begin{proof}
By the H\"older inequality we have
$$
\int_K f d\mu \leq  \|f\|_{L^p(\Om,\mu)}  \mu(K)^{1/q}.
$$
Let $z_0 \in \Om$ be a fixed point and $R:= 2 \diam(\Om)$. Hence,  $ \Om \Subset B:=B(z_0,R)$. For any Borel set $K \subset \Omega$  we  get, by Corollary 5.2 in \cite{Z04} and  Alexander-Taylor's inequality, that
$$ 
\mu(K) \leq C (T_R(K))^{\epsilon/2} \leq C \exp(-\epsilon/2  \  \ca(K,B)^{-1/n}) \leq C \exp(- \epsilon/2  \ca(K,\Om)^{-1/n}),
$$
where $C>0$ depends on $\epsilon$ and  $\diam(\Om)$.
\\
Now, for any $  \tau >1$, we can find $D >0$ depending on $\tau, \epsilon$ and $\diam(\Om)$ such that
$$
 \int_K f d\mu \leq  D \|f\|_{L^p(\Om,\mu)} [\ca(K,\Om)]^{\tau}.
$$
\end{proof}

\section{Preliminaries}

We introduce in this section  basic ingredients of proofs of our results.

We prove in the following proposition that the total mass of Laplacian of the solution is finite when the boundary data is $\Cc^{1,1}$-smooth.

\begin{proposition}\label{mass laplace C11}
Let $\mu$ be a finite Borel measure satisfying Condition $ \mathcal{H}(\tau)$ on $\Omega$ and $\fii \in \Cc^{1,1}(\Omf)$. 
Then the solution $u$ to $Dir(\Om,\fii, d\mu)$ has the property that
$$
\int_\Om \Delta u \leq C,
$$
where $C>0$ depends on $n$, $\Om$ and $\mu(\Om)$.
\end{proposition}

\begin{proof}
Let $u_0$ be the solution to the Dirichlet problem $Dir(\Om,0,d\mu)$.
We first claim that the total mass of $ \Delta u_0$ is finite in $\Om$. 
Indeed, let $ \rho$ be the  defining function of $\Om$. Then by Corollary 5.6 in \cite{Ce04}
we get 
\begin{equation}\label{cegrell inequlaity}
\begin{aligned}
\int_{\Om} dd^c u_0 \wedge (dd^c \rho)^{n-1} & \leq \left(\int_{\Om} (dd^c u_0)^n \right)^{1/n}  \left( \int_{\Om} ( dd^c \rho )^n \right)^{(n-1)/n} \\
& \leq \mu(\Om)^{1/n} \left( \int_{\Om} (dd^c \rho)^n \right)^{(n-1)/n}. \\
\end{aligned} 
\end{equation}  
Since $\Om$ is a bounded strongly hyperconvex Lipschitz domain, there exists a constant $c>0$ such that $ dd^c \rho \geq c \beta$ in $\Omega$. 
Hence,  (\ref{cegrell inequlaity})  yields 
\begin{equation*} 
\begin{aligned}
\int_{\Om} dd^c u_0 \wedge \beta^{n-1}  &  \leq\frac{1}{c^{n-1}}\int_{\Om} dd^c u_0 \wedge               
 (dd^c \rho)^{n-1}  \\
& \leq \frac{\mu(\Om)^{1/n}}{c^{n-1}}  \left( \int_{\Om} (dd^c \rho)^n \right)^{(n-1)/n} .\\ 
\end{aligned}
\end{equation*}
\noindent
Now we note that the total mass of complex Monge-Amp\`ere measure of $\rho$  is finite in $\Om$ by the Chern-Levine-Nirenberg inequality and since $\rho$ is psh and bounded in a neighborhood of $\Omb$.
Therefore, the total mass of $\Delta u_0$ is finite in $\Om$.
\\
Let $\tilde{\fii}$ be a $\Cc^{1,1}$-extension of $\fii$ to $\Omb$ such that
 $ \| \tilde{\fii}\|_{\Cc^{1,1}(\Omb)} \leq C \| \fii \|_{\Cc^{1,1}(\Omf)}$ for some $C>0$.
Now, let $ v= A\rho + \tilde{\fii} + u_0$ where $ A \gg 1$ such that
 $ A \rho +\tilde{\fii} \in PSH(\Om)$. 
By the comparison principle we see that $ v \leq u$ in $\Om$ and $ v=u =\fii$ on $\Omf$. 
Since $ \rho$ is psh in a neighborhood of $ \Omb$ and $ \| \Delta u_0 \|_\Om < +\infty$, 
we deduce that $  \| \Delta v \|_\Om < +\infty$.
This completes the proof.
\end{proof}

\begin{definition}
A finite Borel measure $\mu$  on $\Om$ is said to satisfy Condition $ \mathcal{H}(\infty)$ if for any $ \tau>0$ there exists a positive constant $A$ depending on $\tau$  such that 
$$
\mu(K) \leq A \ca(K,\Om)^{1+\tau},
$$
for any Borel subset $K$ of $\Om$.
\end{definition}

Let $\mu$ be  a measure satisfying Condition $\mathcal{H}(\infty)$, $ 0 \leq f \in L^p(\Om,\mu)$, $p>1$ and $\fii \in \Cc(\Omf)$. Let also $u$ be  the continuous solution to $Dir(\Om,\fii,f d\mu)$ and consider 
$$
u_\delta(z):= \sup_{|\zeta|\leq \delta} u(z+\zeta),\;  z\in \Om_\delta,
$$
where $ \Om_\delta:= \{ z \in \Om ; \text{dist}(z,\Omf) >\delta \}$.

To ensure the H\"older continuity of the solution in $\Om$, we need to control the $L^\infty$-norm of $u_\delta-u$ in $\Om_\delta$.

It is shown in \cite{GKZ08} that  the H\"older norm of the solution $u$  can be estimated by using either $\sup_{\Om_\delta} (u_\delta-u)$ or  $ \sup_{\Om_\delta} (\hat{u}_\delta -u)$, where
$$
\hat u_\delta (z):= \frac{1}{\tau_{2n}\delta^{2n}} \int_{|\zeta -z|\leq \delta} u(\zeta) dV_{2n}(\zeta), \;\; z \in \Omega_\delta,
$$
and $\tau_{2n}$ is the volume of the unit ball in $\C^n$.

It is clear that $\hat u_\delta$ is  defined in $\Om_\delta$, so we extend it with a good control near the boundary $\Omf$. To this end, we assume the existence of $\nu$-H\"older continuous function $v$ such that $v\leq u$ in $\Om$ and $v=u$ on $\Omf$. Then, we present later the construction of such a function.

\begin{lemma}\label{extension}
Let $\Om$ be a bounded strongly hyperconvex Lipschitz domain and $ \fii \in \Cc^{0,\alpha}(\Omf)$, $0<\alpha \leq 1$.
Assume that there is a function $v \in \Cc^{0,\nu}(\Omb)$ for $0 <\nu \leq 1$, such that 
$ v \leq u$ in $\Om$ and $ v= \fii$ on $\Omf$. Then there exist $\delta_0>0$ small enough and $ c_0>0$, depending on $\Om,\| \fii \|_{\Cc^{0,\alpha}(\Omf)}$ and $\|v\|_{\Cc^{0,\nu}(\Omb)} $,   such that for any $  0< \delta_1 \leq \delta < \delta_0$ the function
\begin{equation*}
 	\tilde {u}_{\delta_1}
 		= \begin{cases}
			\max\{\hat{u}_{\delta_1}, u+ c_0 \delta^{\nu_1} \} & \text{ in } \; \Om_\delta, \\
                   u +c_0 \delta^{\nu_1} & \text{ in } \; \Om\setminus \Omega_\delta,
           \end{cases}
\end{equation*}
is  plurisubharmonic  in $\Omega$ and  continuous on $\Omb$, where $\nu_1= \min \{ \nu, \alpha/2\}$.
\end{lemma}

\begin{proof}
If we  prove that $ \hat{u}_{\delta_1} \leq u+ c_0 \delta^{\nu_1}$ on $\partial \Om_\delta$, then the required result can be obtained by the standard gluing procedure.
\\
Thanks to Corollary 4.6 in \cite{Ch15}, we find a plurisuperharmonic function $\tilde{v} \in \Cc^{0,\alpha/2}(\Omb)$ such that $ \tilde{v}=\fii$ on $\Omf$ and 
$$
\| \tilde{v}\|_{\Cc^{0,\alpha/2}(\Omb)} \leq C \|\fii\|_{\Cc^{0,\alpha}(\Omf)},
$$
where $C$ depends on $\Om$.
From the maximum principle we see that $u \leq  \tilde{v}$ in $\Om$ and $ \tilde{v}=u$ on $\Omf$.
\\
Fix  $z\in \partial \Om_\delta$,
there exists $\zeta \in \C^n$ with $\|\zeta\|=\delta_1$ such that
 $\hat{u}_{\delta_1}(z) \leq u(z+\zeta)$. Hence we obtain
$$
\hat{u}_{\delta_1} (z)-u(z) \leq u(z+\zeta)-u(z) \leq \tilde{v}(z+\zeta) -v(z).
$$
We  choose $\zeta_0 \in \C^n$, with $\|\zeta_0\| = \delta$, such that
$z+\zeta_0 \in \Omf$.
Since  $\tilde{v}(z+\zeta_0) = v(z+\zeta_0)$, we get
\begin{equation*}
\begin{aligned}
	\tilde{v}(z+\zeta) - v (z)
	\leq \;& [\tilde{v}(z+\zeta) -\tilde{v}(z+\zeta_0)] + [v(z+\zeta_0) - v(z)] \\
	\leq \;& 2\|\tilde{v}\|_{\Cc^{0,\alpha/2}(\Omb)} \delta^{\alpha/2} + \|v\|_{\Cc^{0,\nu}(\Omb)} \delta^\nu \\
	\leq \;& c_0 \delta^{\nu_1}, \\
\end{aligned}
\end{equation*}
where $c_0:= 2 C \| \fii \|_{\Cc^{0,\alpha}(\Omf)}+ \| v \|_{\Cc^{0,\nu}(\Omb)} >0$.
\end{proof}

Moreover, we can conclude from the last argument that
\begin{equation}\label{boundary holder}
|u(z_1)-u(z_2)| \leq 2 c_0 \delta^{\nu_1},
\end{equation}
for all $z_1,z_2 \in \Omb \setminus \Om_\delta$ such that $|z_1-z_2|\leq \delta$.

\begin{remark}\label{Rem.exten.}
When $\fii\in \Cc^{1,1}(\Omf)$, the last lemma holds for $\nu_1=\nu$. Indeed, let $\tilde{\fii}$ be a $\Cc^{1,1}$-extension of $\fii$ to $\Omb$. We define the plurisuperharmonic Lipschitz function 
  $ \tilde{v}:= -A \rho +\fii$, where $ A \gg 1$ and $\rho$ is the defining function of $\Om$. Hence, the constant $c_0$ in Lemma \ref{extension} will  depend only on $\Om$, $\|\fii \|_{\Cc^{1,1}(\Omf)}$ and $\| v\|_{\Cc^{0,\nu}(\Omb)}$.
\end{remark}

The following weak stability estimate, proved in \cite{GKZ08} for the Lebesgue measure, plays an important role in our work. A similar, but weaker, estimate was established by  Ko{\l}odziej \cite{Ko02} and in the compact setting it was proved by Eyssidieux, Guedj and Zeriahi \cite{EGZ09}. This estimate is still true for any  finite Borel measure $ \mu$ satisfying Condition $\mathcal{H}(\infty)$.

\begin{theorem}\label{stability theorem}
Let $\mu$ be a  finite Borel measure on $\Om$ satisfying Condition $ \mathcal{H}(\infty)$ and
 $0 \leq f \in L^p(\Omega,\mu)$, $p>1$.
Suppose that $v_1,v_2$ are two bounded psh functions in $\Omega$ such that $ \liminf_{z \to \Omf} (v_1-v_2)(z) \geq 0$ and $ (dd^c v_1)^n = f d\mu$.
Fix $r\geq 1$ and  $0 < \gamma <r/(nq+r)$, $1/p+1/q=1$. Then there exists a constant $c_1=c_1(r,\gamma,n,q)>0$ such that 
\begin{equation}\label{stab}
\sup_{\Omega} (v_2-v_1) \leq c_1 (1+ \|f\|^\eta_{L^p(\Om,\mu)}) ||(v_2-v_1)_+||^\gamma_{L^r(\Omega,\mu)},
\end{equation}
where $ (v_2-v_1)_+ =\max(v_2-v_1,0)$ and $ \eta= \frac{1}{n}+ \frac{\gamma q}{r-\gamma(r+nq)}$.
\end{theorem}

\section{Proofs of main results}

\begin{theorem}\label{via barrier}
Let $\Om$ be a bounded strongly hyperconvex Lipschitz domain and let $\mu$ be a  finite Borel measure on $\Om$ satisfying Condition  $\mathcal{H}(\infty)$. Suppose that $\fii \in \Cc^{0,\alpha}(\Omf)$, $ 0<\alpha \leq 1$, and $ 0\leq f \in L^p(\Om,\mu)$ for $p>1$.
Then the solution $u$ to $Dir(\Om,\fii,f d\mu)$ is H\"older continuous on $\Omb$ of exponent  $\frac{1}{\lambda}\min \{ \nu, \alpha/2, \tau \gamma\}$, for any $ \gamma < 1/(nq+1)$ and $1/p+1/q=1$, 
if the two following conditions hold: 
\\
(i) there exists  $v \in \Cc^{0,\nu}(\Omb)$, for $0 <\nu \leq 1$, such that $ v \leq u$ in $\Om$ and $ v= \fii$ on $\Omf$,
\\
(ii) and  $ \| \hat{u}_{\delta_1}-u\|_{L^1(\Om_\delta, \mu)} \leq c \delta^{\tau}$, where $c,\tau>0$ and $ 0<\delta_1 =\delta ^\lambda$, for some $\lambda\geq 1$.

Moreover, if $\fii \in \Cc^{1,1}(\Omf)$ then the H\"older exponent of $u$ will be  $\frac{1}{\lambda}\min \{ \nu,  \tau \gamma\}$.
\end{theorem}

\begin{proof}
It follows from Lemma \ref{extension} that there exist $c_0>0$ and $\delta_0>0$ so that 
\begin{equation*}
 	\tilde {u}_{\delta_1}
 		= \begin{cases}
			\max\{\hat{u}_{\delta_1}, u+ c_0 \delta^{\nu_1} \} & \text{ in } \; \Om_\delta, \\
                   u + c_0 \delta^{\nu_1} & \text{ in } \; \Om\setminus \Omega_\delta,
           \end{cases}
\end{equation*}
belongs to $PSH(\Om)\cap \Cc(\Omb)$, for $ 0<\delta_1 \leq \delta < \delta_0$ and $\nu_1=\min\{ \nu,\alpha/2\}$.
\\
By applying Theorem \ref{stability theorem} with $v_1 := u+c_0 \delta^{\nu_1}$ and $v_2 :=  \tilde{u}_{\delta_1}$,  we infer that
$$
\sup_{\Om_\delta} (\hat{u}_{\delta_1}-u-c_0 \delta^{\nu_1}) \leq \sup_\Om (\tilde{u}_{\delta_1} -u -c_0 \delta^{\nu_1}) \leq c_1 (1+\|f\|_{L^p(\Om,\mu)}^{\eta}) \| (\tilde{u}_{\delta_1} -u -c_0 \delta^{\nu_1})_+ \|_{L^1(\Om,\mu)}^{\gamma},
$$
where  $\eta := 1/n + \gamma q/[1-\gamma(1+nq)]$, $c_1= c_1(n,q,\gamma)$ and $0 < \gamma < 1/(nq+1)$ is fixed.
Since $ \tilde{u}_{\delta_1} = u+ c_0 \delta^{\nu_1}$ in $\Om \setminus \Om_\delta$ and 
$$
\|(\tilde u_{\delta_1} -u - c_0 \delta^{\nu_1})_+\|_{L^1(\Om,\mu)} \leq \| \hat{u}_{\delta_1} -u \|_{L^1(\Om_\delta, \mu)},
$$
 we conclude that
$$
\sup_{\Om_\delta} (\hat{u}_{\delta_1}-u) \leq c_0 \delta^{\nu_1} + c_1 (1+\|f\|_{L^p(\Om,\mu)}^{\eta}) \| \hat{u}_{\delta_1} -u  \|_{L^1(\Om_\delta,\mu)}^{\gamma}.
$$
By hypotheses we have
$$
\sup_{\Om_\delta} (\hat{u}_{\delta_1} -u) \leq c_0 \delta^{\nu_1} + c_1 c^\gamma (1+\|f\|_{L^p(\Om,\mu)}^{\eta}) \delta^{\tau \gamma}.
$$
Let us set $c_2:= (c_0+c_1 c^\gamma)(1+\|f\|_{L^p(\Om,\mu)}^{\eta})$. We derive from the last inequality that
$$
\sup_{\Om_\delta} (\hat{u}_{\delta_1}-u) \leq c_2 \delta^{\min\{ \nu_1, \tau \gamma \}}.
$$
This means that 
$$ 
\hat{u}_\delta - u \leq c_2 \delta^{\frac{1}{\lambda}\min\{ \nu_1, \tau \gamma \}}  \text{ in } \Om_{\delta^{1/\lambda}}.
$$
Hence, by Lemma 4.2 in \cite{GKZ08}, there exists $c_3,\tilde{\delta}_0>0$ such that for all $0<\delta<\tilde{\delta}_0$ we have
\begin{equation}\label{holder int}
u_\delta -u \leq c_3 \delta^{\frac{1}{\lambda}\min\{ \nu_1, \tau \gamma \}} \text{ in } \Omega_{\delta^{1/\lambda}}.
\end{equation}

Thus, (\ref{holder int}) and  (\ref{boundary holder}) yield the H\"older continuity of $u$ on $\Omb$ of exponent $\frac{1}{\lambda}\min \{ \nu, \alpha/2, \tau \gamma\}$,  for any $ \gamma < 1/(nq+1)$ and $1/p+1/q=1$.

Finally, if $\fii \in \Cc^{1,1}(\Omf)$, we get that the H\"older exponent is 
$\frac{1}{\lambda}\min \{ \nu,  \tau \gamma\}$, since $\nu_1=\nu$  (see Remark \ref{Rem.exten.}).
\end{proof}

The first step in the proof of Theorem \ref{main2 C11} is to  estimate  $ \|\hat{u}_{\delta} - u \|_{ L^1(\Om_\delta,\mu)}$, so we present the following lemma.

\begin{lemma}\label{Laplacian2}
Let $\Om \subset \C^n$ be a strongly hyperconvex Lipschitz domain and let $\mu$ be a Hausdorff-Riesz measure of order $2n-2+\epsilon$ on $\Om$ for $ 0< \epsilon \leq 2$. Suppose that $ \fii \in \Cc^{1,1}(\Omf)$ and $f \in L^p(\Om,\mu)$, $p>1$, then the solution $u$ to the Dirichlet problem $Dir(\Om,\fii,f d\mu)$ satisfies
$$
\int_{\Omega_{\delta}} [\hat{u}_{\delta}(z)- u(z)] d \mu (z) \leq C \delta^\epsilon, 
$$
where  $C $ is a positive  constant depending on $n$, $\epsilon$, $\Om$, $\|f\|_{L^p(\Om,\mu)}$ and $ \mu(\Om) $.
\end{lemma}

\begin{proof}
Let us denote by $\sigma_{2n-1}$ the surface measure of the unit sphere.
It follows from the Poisson-Jensen formula, for $ z \in \Om_\delta$ and $0<r< \delta$, that
$$
\frac{1}{\sigma_{2n-1} r^{2n-1}} \int_{\partial B(z,r)} u(\xi) d\sigma(\xi) -u(z) = c_n \int_0^r t^{1-2n} \left( \int_{B(z,t)} \Delta u(\xi) \right) dt.
$$
Using  polar coordinates we obtain for $ z \in \Om_\delta$,
$$
\hat u_{\delta} (z) - u (z) = \frac{c_n}{ \delta^{2n}} \int_{0}^{\delta} r^{2n-1}dr \int_{0}^r t^{1-2n} dt \left( \int_{B(z,t)} \Delta u(\xi) \right) .
$$
Now, we integrate on $\Om_\delta$ with respect to $d\mu$ and use Fubini's theorem
\begin{equation*}
\begin{aligned}
\int_{\Omega_{\delta}} [\hat u_{\delta} (z) - u (z)] d\mu (z) & = \frac{c_n}{ \delta^{2n}} \int_{\Om_\delta}  \int_{0}^{\delta} r^{2n-1}dr \int_{0}^r t^{1-2n} dt \left( \int_{|\xi-z|\leq t} \Delta u(\xi)  \right) d\mu(z)\\
& = \frac{c_n}{ \delta^{2n}}   \int_{0}^{\delta} r^{2n-1}dr \int_{0}^r t^{1-2n} dt \int_{\Om_\delta} \left( \int_{B(z,t)} \Delta u(\xi)  \right) d\mu(z)\\
& \leq \frac{c_n}{ \delta^{2n}}   \int_{0}^{\delta} r^{2n-1}dr \int_{0}^r t^{1-2n} dt \int_{\Om} \left( \int_{B(\xi,t)} d\mu(z)  \right) \Delta u(\xi)  \\
& \leq \frac{c_n}{\epsilon(2n+\epsilon)} \left( \int_\Om \Delta u \right) \delta^\epsilon .\\
\end{aligned}
\end{equation*}
By Proposition \ref{mass laplace C11}, the total mass of $\Delta u$ is finite in $\Om$ and this completes the proof. 
\end{proof}

When $\fii$ is not $\Cc^{1,1}$-smooth, the measure $\Delta u$ may have infinite mass on $\Omega$. 
Fortunately, we can estimate  $\|\hat u_{\delta_1}  - u \|_{L^1(\Om_\delta,\mu)}$ for some $\delta_1 < \delta \leq 1$.
 
\begin{lemma}\label{Laplacian3}
Let $\Om \subset \C^n$ be a strongly hyperconvex Lipschitz domain and let $\mu$ be a Hausdorff-Riesz measure of order $2n-2+\epsilon$ on $\Om$, for $ 0< \epsilon \leq 2$. Suppose that $ 0 \leq f \in L^p(\Om,\mu)$, $p>1$ and $\fii \in \Cc^{0,\alpha}(\Omf)$, $\alpha \leq 1$.
Then for any small $\epsilon_1>0$, we have the following inequality 
$$
\int_{\Omega_{\delta}} [\hat u_{\delta_1} (z) - u (z)] d\mu (z) \leq C \delta^{\epsilon/2- \epsilon_1}, 
$$
where   $ \delta_1= (1/2)\delta^{1/2+3/\epsilon}$ and $C$ is a positive  constant depending on $n$, $\Om$, $\epsilon$, $\epsilon_1$ and  $\| u \|_{L^\infty(\Omb)}$.
\end{lemma}

\begin{proof}
One sees as in the proof of Lemma \ref{Laplacian2} that
$$
\hat u_{\delta_1} (z) - u (z) = \frac{c_n}{ \delta_1^{2n}} \int_{0}^{\delta_1} r^{2n-1}dr \int_{0}^r t^{1-2n} dt \left( \int_{B(z,t)} \Delta u(\xi) \right) .
$$
Then, we integrate on $\Om_\delta$ with respect to $\mu$ and use Fubini's Theorem

\begin{equation*}
\begin{aligned}
\int_{\Omega_{\delta}} [\hat u_{\delta_1} (z) - u (z)] d\mu (z) &  = \frac{c_n}{\delta_1^{2n}}   \int_{0}^{\delta_1} r^{2n-1}dr \int_{0}^r t^{1-2n} dt \int_{\Om_\delta} \left( \int_{B(z,t)} \Delta u(\xi) \right) d\mu(z)\\
& \leq \frac{c_n}{\delta_1^{2n}}   \int_{0}^{\delta_1} r^{2n-1}dr \int_{0}^r t^{1-2n} dt \int_{\Om_{\delta-t}} \left( \int_{B(\xi,t)} d\mu(z)  \right) \Delta u(\xi) \\
& \leq \frac{c_n}{\delta_1^{2n}}   \int_{0}^{\delta_1} r^{2n-1}dr \int_{0}^r t^{-1+\epsilon} dt \int_{\Om_{\delta-t}} \Delta u(\xi) \\
& \leq \frac{c_n}{\delta_1^{2n}}  \sup_{\Om_{\delta- \delta_1}} (-\tilde\rho)^{-(3+\epsilon_1)/2} \int_{0}^{\delta_1} r^{2n-1}dr \int_{0}^r t^{-1+\epsilon} dt \int_{\Om_{\delta-t}} (-\tilde\rho)^{(3+\epsilon_1)/2} \Delta u(\xi) \\
& \leq \frac{c_n}{\delta_1^{2n}}  \sup_{\Om_{\delta/2}} (-\tilde\rho)^{-(3+\epsilon_1)/2} \|(-\tilde\rho)^{(3+\epsilon_1)/2} \Delta u \|_{\Om} \int_{0}^{\delta_1} r^{2n-1}dr \int_{0}^r t^{-1+\epsilon} dt \\
& \leq \frac{4 c_n \delta^{-3-\epsilon_1}}{\epsilon(2n+\epsilon)} \delta_1^{\epsilon} \|(-\tilde\rho)^{(3+\epsilon_1)/2} \Delta u \|_{\Om}  \\
& \leq  C_1  \delta^{ \epsilon/2 -\epsilon_1} \|(-\tilde\rho)^{(3+\epsilon_1)/2} \Delta u \|_{\Om}, \\
\end{aligned}
\end{equation*}
where $\tilde \rho \in PSH(\Om)\cap \Cc^{0,1}(\Omb) $ is as in Lemma \ref{lemma tilde rho} and $C_1$ is a positive constant depending on $n$ and $\epsilon$.
To complete the proof we demonstrate that the mass $\|(-\tilde\rho)^{(3+\epsilon_1)/2} \Delta u \|_{\Om}$ is finite.
The following idea is due to \cite{BKPZ15} with some appropriate modifications. We set for simplification $ \theta:=(3+\epsilon_1)/2$.
Let $\rho_\eta$ be the standard  regularizing kernels 
with supp $ \rho_\eta \subset B(0,\eta) $ and 
$ \int_{B(0,\eta)} \rho_\eta dV_{2n} = 1 $. Hence,  $ u_\eta = u* \rho_\eta \in \Cc^\infty \cap PSH(\Om_\eta) $ decreases to $u$ in $\Omega$. It is clear that $\|u_\eta \|_{L^\infty(\Om_\eta)} \leq \|u\|_{L^\infty(\Om)}$ and  the first dirivatives of $u_\eta$ have  $L^\infty$-norms less than $\|u\|_{L^\infty(\Om)}/\eta $. We denote by $\chi_{\Om_\eta}$ the characteristic function of $\Om_\eta$ and by $\rho $ the defining function of $\Om$.  Since $u_\eta \searrow u$ in $\Om$, we have $ \chi_{\Om_\eta} (-\tilde\rho)^{\theta} \Delta u_\eta $ converges to $(-\tilde\rho)^{\theta} \Delta u$ in the weak sense of measures.

It is sufficient to show  that 
$$
I:= \int_{\Om_\eta} (-\tilde\rho)^{\theta} \Delta u_\eta,
$$
is bounded by an absolute constant independent of $\eta$. 
We have by Stokes' theorem
\begin{equation*}
\begin{aligned}
I  & = \int_{\Omf_\eta} (-\tilde\rho)^{\theta} d^c u_\eta \wedge \beta^{n-1} +\theta \int_{\Om_\eta} (-\tilde\rho)^{\theta-1} d\tilde\rho \wedge d^c u_\eta \wedge \beta^{n-1}  .  \\
\end{aligned}
\end{equation*}
Note that 
\begin{equation*}
\begin{aligned}
\int_{\Omf_\eta} (-\tilde\rho)^{\theta-1} u_\eta d^c \tilde\rho \wedge \beta^{n-1} & = \int_{\Omega_\eta} (-\tilde\rho)^{\theta-1} d u_\eta \wedge d^c \tilde\rho \wedge \beta^{n-1} + \\
&+ \int_{\Om_\eta} (-\tilde\rho)^{\theta-1}  u_\eta  dd^c \tilde\rho \wedge \beta^{n-1}  \\
& - (\theta-1) \int_{\Om_\eta} (-\tilde\rho)^{\theta -2}  u_\eta  d \tilde\rho \wedge d^c \tilde\rho \wedge \beta^{n-1}.\\
\end{aligned}
\end{equation*}
Hence, we get
\begin{equation*}
\begin{aligned}
I & = \int_{\Omf_\eta} (-\tilde\rho)^{\theta} d^c u_\eta \wedge \beta^{n-1}  +\theta \int_{\Omf_\eta} (-\tilde\rho)^{\theta-1} u_\eta d^c \tilde\rho \wedge \beta^{n-1}\\
& - \theta \int_{\Om_\eta} (-\tilde\rho)^{\theta-1}  u_\eta  dd^c \tilde\rho \wedge \beta^{n-1} + \theta(\theta-1)  \int_{\Om_\eta} (-\tilde\rho)^{\theta-2}  u_\eta  d \tilde\rho \wedge d^c \tilde\rho \wedge \beta^{n-1} \\
& \leq C \|u\|_{L^\infty(\Omb)} \left( \int_{\Omf_\eta} d \sigma +\int_{\Om_\eta} dd^c \tilde\rho \wedge \beta^{n-1}+\int_{\Om_\eta} (-\tilde\rho)^{\theta-2}   \beta^{n} \right), \\
& \leq C \|u\|_{L^\infty(\Omb)} \left( \int_{\Omf_\eta} d \sigma +\int_{\Om} dd^c \rho \wedge \beta^{n-1}+\int_{\Om} (-\tilde\rho)^{(-1+\epsilon_1)/2}   \beta^{n} \right), \\
\end{aligned}
\end{equation*}
where $d \sigma= d^c \rho \wedge (dd^c \rho )^{n-1}$ and $\rho$ is the defining function of $\Om$. Since $\rho$ is psh in a neighborhood of $\Omb$, the second integral in the last inequality is finite. Thanks to Lemma \ref{lemma tilde rho}, we have $-\tilde\rho \geq  \dist(.,\Omf)^2$ near $\Omf$ and so the third integral will be finite since $\epsilon_1>0$ small enough. Consequently,  we infer that $I$ is bounded by a constant  independent of $\eta$ and then this proves our claim.
\end{proof}

\begin{corollary}\label{Laplacian 3 strongly}
When $\Omega$ is a smooth strongly pseudoconvex domain, then Lemma \ref{Laplacian3} holds also for $ \delta_1 = (1/2) \delta^{1/2+1/\epsilon}$.
\end{corollary}
\begin{proof}
Let $\rho$ be the defining function of $\Om$. In view of  Remark \ref{rho smooth} and the last argument, we can estimate $\|(-\rho)^{1+\epsilon_1} \Delta \U \|_{\Om}$, for $\epsilon_1>0$ small enough. So the proof of Lemma \ref{Laplacian3} is still true for  $\delta_1:= (1/2) \delta^{1/2+1/\epsilon}$.
\end{proof}

We are now in a position to prove the main theorems. We begin to prove the H\"older continuity of the solution to $Dir(\Om,\fii,fd\mu)$ where $\mu$ is a Hausdorff-Riesz measure of order $2n-2+\epsilon$ and $\fii\in \Cc^{1,1}(\Omf)$.

\begin{proof}[\bf Proof of Theorem \ref{main2 C11}.]
We first assume that $f$ equals to zero near the boundary $\Omf$, that is, there exists a compact $ K \Subset \Om$ such that $f=0$ on $\Omega \setminus K$.
Since $\fii \in \Cc^{1,1}(\Omf)$, we extend it to $\tilde{\fii} \in \Cc^{1,1}(\Omb)$ such that 
$\|\tilde{\fii} \|_{\Cc^{1,1}(\Omb)} \leq C \| \fii\|_{\Cc^{1,1}(\Omf)}$ for some constant $C$. Let $\rho$ be the defining function of $ \Om$ and let $ A \gg 1$ be so that $ v:= A \rho +\tilde{\fii} \in PSH(\Om)$ and $ v \leq u $ in a neighborhood of $K$. Moreover, by the comparison principle, we see that $ v \leq u $ in $\Om \setminus K$. Consequently, $ v \in PSH(\Om) \cap \Cc^{0,1}(\Omb)$ and satisfies  $ v \leq u $ on $\Omb$ and $ v=u=\fii$ on $\Omf$.
 It follows from Theorem \ref{via barrier} and Lemma \ref{Laplacian2} that  $ u \in \Cc^{0,\epsilon \gamma}(\Omb)$, for any  $ 0< \gamma < 1/(nq+1)$.
 
In the general case, fix a large ball $ B \subset \C^n$ containing $ \Omega $ and define a function $\tilde{f} \in L^p(B,\mu)$ so that $ \tilde{f}:= f $ in $\Omega$ and $ \tilde{f}:=0$ in $ B \setminus \Omega$. Hence, the solution to the following Dirichlet problem
\begin{center}
$ \left\{\begin{array}{ll}
            v_1\in PSH(B)\cap\Cc(\bar{ B}),& \\
           (dd^c v_1)^n=\tilde{f} d\mu  & \text{  in  } \; B, \\
            v_1=0  & \text{ on  }  \partial B, \\
         \end{array} \right. $

\end{center}
belongs to $\Cc^{0,\gamma'}(\bar{B})$, with  $\gamma' =  \epsilon \gamma$ for any $\gamma < 1/(nq+1)$.
\\
Let $ h_{\fii-v_1}$ be the continuous solution  to $Dir(\Om, \fii-v_1, 0)$. Then, Theorem A in \cite{Ch15} implies that $ h_{\fii-v_1}$ belongs to $\Cc^{0,\gamma'/2}(\Omb)$.

This enables us to construct a H\"older barrier for our problem. 
We take $ v_2= v_1 + h_{\fii-v_1}$. It is clear that
$ v_2 \in PSH(\Om)\cap \Cc^{0,\gamma'/2}(\Omb)$ and 
 $ v_2 \leq u$ on $\Omb$ by the comparison principle. 
Hence, Theorem \ref{via barrier} and  Lemma \ref{Laplacian2}  imply that the solution $u$ to $Dir(\Om,\fii,f d\mu)$ is H\"older continuous on $\Omb$ of exponent $ \epsilon \gamma/2$ for any $ 0< \gamma < 1/(nq+1)$.
\end{proof}

\begin{proof}[\bf Proof of Theorem \ref{main2 C0,alpha}.]
Let $v_1$ be as in the proof of Theorem \ref{main2 C11} and $h_{\fii - v_1}$ be the solution to $Dir(\Om,\fii-v_1,0)$.
In order to apply Theorem \ref{via barrier}, we set $ v= v_1 + h_{\fii-v_1}$. 
Hence, $ v \in PSH(\Om) \cap \in \Cc(\Omb)$, $ v= \fii $ on $\Omf$ and $ (dd^c v)^n \geq f d\mu$ in $\Om$. The comparison principle yields $ v \leq u $ in $\Om$. 
Moreover, By Theorem A in \cite{Ch15}, we have $ h_{\fii-v_1} \in \Cc^{0,\gamma''}(\Omb)$ with $ \gamma''= 1/2 \min \{ \alpha, \epsilon \gamma \}$.
Hence, it stems from Theorem \ref{via barrier} and Lemma \ref{Laplacian3} that the solution $ u$ is H\"older continuous on $\Omb$ of exponent $ \frac{\epsilon}{\epsilon+6} \min \{ \alpha, \epsilon \gamma \}$, for any $ 0< \gamma < 1/(nq+1)$.

Moreover, when $\Om$ is a smooth strongly pseudoconvex domain, we get using Theorem \ref{via barrier} and Corollary \ref{Laplacian 3 strongly} that the solution is   H\"older continuous with better exponent  $\frac{\epsilon}{\epsilon+2} \min \{ \alpha, \epsilon \gamma \}$, for any $ 0< \gamma < 1/(nq+1)$.
\end{proof}

\begin{corollary}
Let $\Om$ be a finite intersection of smooth strongly pseudoconvex domains in $\C^n$. Assume that $\fii \in \Cc^{0,\alpha}(\Omf)$, $0<\alpha \leq 1$,  and $ 0 \leq f \in L^p(\Om)$ for some $p>1$. Then the solution $u$ to the Dirichlet problem $Dir(\Om,\fii,f dV_{2n})$ belongs to $ \Cc^{0,\alpha'}(\Omb)$, with $ \alpha' =  \min \{ \alpha/2, \gamma \}$, for any $\gamma < 1/(nq+1)$ and $1/p+1/q=1$. 
\\
Moreover, if $ \fii \in \Cc^{1,1}(\Omf)$, then the solution $u$ is $\gamma$-H\"older continuous on $\Omb$.
\end{corollary}

The first part of this corollary was proved in Theorem 1.2 in \cite{BKPZ15} with the H\"older exponent $\min\{ 2\gamma,\alpha\}\gamma$ and  the second part  was proved in \cite{GKZ08} and \cite{Ch15} (see also \cite{N14,Ch14} for the complex Hessian equation).

\smallskip
Our final purpose concerns how to get the H\"older continuity of the solution to the Dirichlet problem $Dir( \Om,\fii, f d \mu)$, by means of the H\"older continuity of  subsolutions to  $Dir(\Om,\fii,d\mu)$ for some special measure $ \mu $ on $\Om$.

\begin{theorem}\label{main3}
Let $\mu$ be a finite Borel measure on a bounded strongly hyperconvex Lipschitz domain $\Om$. Let also $ \fii \in \Cc^{0,\alpha}(\Omf)$, $ 0< \alpha \leq 1$ and $ 0\leq f \in L^p(\Om,\mu)$, $p>1$.
Assume that there exists a $\lambda$-H\"older continuous plurisubharmonic function  $w$ in $\Omega$
such that $(dd^c w)^n \geq \mu$. If, near the boundary, $\mu$ is Hausdorff-Riesz of order $2n-2+\epsilon$ for some $0<\epsilon\leq 2$, then the solution $u$ to $Dir(\Om,\fii, fd\mu)$  is H\"older continuous on $\Omb$. 
\end{theorem}

\begin{proof}
Let $\Om_1 \Subset \Om$ be an open set such that $\mu$ is a Hausdorff-Riesz measure on $\Om \setminus \Om_1$. Let also $\tilde{\mu}$ be a Hausdorff-Riesz measure of order $2n-2+\epsilon$ so that $\tilde{\mu}$ equals $\mu$ in $\Omega \setminus \Om_1$.

As $\fii$ is not $\Cc^{1,1}$-smooth, we can not control $\| \hat{u}_{\delta} -u \|_{L^1(\Om_\delta, \mu)}$ as in the proof of Lemma \ref{Laplacian2}. Thus, we will estimate $\| \hat{u}_{\delta_1} -u \|_{L^1(\Om_\delta, \mu)}$ with  $ \delta_1:= (1/2)\delta^{1/2+3/\epsilon}$. 

We have
$$
\int_{\Om_\delta} (\hat{u}_{\delta_1} -u) d\mu \leq \int_{\Om_1} (\hat{u}_{\delta_1} -u) d\mu +\int_{\Om_\delta } (\hat{u}_{\delta_1} -u) d \tilde \mu.
$$
Fix $\epsilon_1>0$ small enough and let $U \Subset \Omega$ be a a neighborhood of $ \bar{\Om}_1$.
Then,  Theorem 4.3 in \cite{DDGPKZ14} and Lemma \ref{Laplacian3}  yield that
\begin{equation*}
\begin{aligned}
\int_{\Om_\delta} (\hat{u}_{\delta_1} -u) d\mu & \leq   \int_{\Om_1} (\hat{u}_{\delta_1} -u) (dd^c w)^n + \int_{\Om_\delta } (\hat{u}_{\delta_1} -u) d \tilde \mu \\
& \leq C_1 \| \Delta u\|_{U} \delta_1^\frac{2\lambda}{\lambda+ 2n} +  C_2 \delta^{\epsilon/2 - \epsilon_1}, \\
\end{aligned}
\end{equation*}
where $ C_1=C_1(\Om_1,U)$ is a positive constant and $C_2$ depends on $n$, $\Om$, $\epsilon$, $\epsilon_1$ and $\|u\|_{L^\infty(\Omb)}$. 
Since the mass of $\Delta u$ is locally bounded, there exists a constant $C_3>0$ such that 
$$
\int_{\Om_\delta} (\hat{u}_{\delta_1} -u) d\mu \leq  C_3  \delta^{\tau},
$$
where 
$\tau =\min\{ \frac{\epsilon}{2}- \epsilon_1, \frac{\lambda(\epsilon+6)}{\epsilon(\lambda +2n)} \}$.
\\
The last requirement to apply Theorem \ref{via barrier} is to construct a function 
$ v \in \Cc^{0,\nu}(\Omb)$ for $ 0<\nu \leq 1$ such that $ v \leq u $ in $\Om$ and $ v=\fii$ on $\Omf$.  Let us denote by $w_1$ the solution to $Dir(\Om,0,f d\tilde{\mu})$ and $ h_{\fii}$  the solution to $Dir(\Om,\fii,0)$.
Now, set $ v= w_1+ h_{\fii}+ A\rho$ with $A \gg 1$ so that $v \leq u $ in a neighborhood of $\Omb_1$. It is clear that $ v \in PSH(\Om) \cap \Cc(\Omb)$,  $ v=\fii$ on $\Omf$ and $ v \leq u $ in $\Om$ by the comparison principle.
Moreover, by Theorem A in \cite{Ch15} and Theorem \ref{main2 C11}, we infer that $ v \in \Cc^{0,\nu}(\Omb)$, for $ \nu= 1/2 \min \{ \epsilon \gamma, \alpha\}$ and any $\gamma < 1/(nq+1)$.
Finally, we get from Theorem \ref{via barrier} that $u$ is H\"older continuous on $\Omb$ of exponent $ \frac{\epsilon}{\epsilon+6} \min\{ \alpha, \epsilon \gamma,  \frac{2\lambda \gamma (\epsilon+6)}{\epsilon (\lambda+ 2n)} \}$.
\end{proof}

The following are nice applications of Theorem \ref{main3}.

\begin{corollary}\label{main4}
Let $\Om \subset \C^n$ be a bounded strongly hyperconvex Lipschitz domain and  let $\mu$ be a finite Borel measure with compact support on $\Om$. Let also $ \fii \in \Cc^{0,\alpha}(\Omf)$, $ 0< \alpha \leq 1$ and $0\leq f \in  L^p(\Om,\mu)$, $p>1$. Assume that there exists a $\lambda$-H\"older continuous psh function  $w$ in $\Omega$ such that  $ (dd^c w)^n \geq \mu $. Then the solution to the Dirichlet problem $Dir(\Om,\fii, f d\mu)$ is H\"older continuous on $\Omb$ of exponent $ \min\{ \frac{\alpha}{2},  \frac{2\lambda \gamma}{\lambda+ 2n} \}$, for any $ \gamma < 1/(nq+1)$ and $1/p+1/q=1$.   
\end{corollary}

\begin{example}
Let $\mu$ be a finite Borel measure with compact support on  a  bounded strongly hyperconvex Lipschitz domain $\Om$. Let also  $ \fii \in \Cc^{0,\alpha}(\Omf)$, $ 0< \alpha \leq 1$ and $0\leq f \in  L^p(\Om,\mu)$, $p>1$. Suppose that $\mu \leq dV_n$, where $dV_n$ is  the Lebesgue measure of the totally real part $\R^n$ of $\C^n$, then the solution to the Dirichlet problem $Dir(\Om,\fii, f d\mu)$ is H\"older continuous on $\Omb$ of exponent  $ \min\{ \frac{\alpha}{2}, \frac{2 \gamma}{1+2n} \}$, for any $ \gamma < 1/(nq+1)$ and $1/p+1/q=1$.
\end{example}

Indeed, since $ \R^n=\{ Im z_j=0, j=1,...,n \}$, one can present the Lebesgue measure of the totaly real part $\R^n$ of $\C^n$ in the form
$$
 \left(dd^c \sum_{j=1}^n (Imz_j)_+ \right)^n.
$$
Let us set $ w= \sum_{j=1}^n (Imz_j)_+$. It is clear that  $w \in PSH(\Om) \cap \Cc^{0,1}(\Omb)$ and $ \mu \leq (dd^c w)^n$ on $\Om$.  Corollary \ref{main4} yields that the solution $u$ belongs to $\Cc^{0,\alpha'}(\Omb)$ with $\alpha'= \min\{ \alpha/2, \frac{2 \gamma}{1+ 2n}  \}$, for any $ \gamma < 1/(nq+1)$.

\bigskip

\noindent 
Mohamad Charabati \\
Institut de Math\'ematiques de Toulouse \\
Universit\'e Paul Sabatier \\
118 route de Narbonne \\
31062 Toulouse Cedex 09 (France). \\
E-mail: {\tt mohamad.charabati@math.univ-toulouse.fr }

\end{document}